\documentclass[11 pt, reqno]{amsart}
\usepackage{amsmath}
\usepackage[alphabetic]{amsrefs}
\usepackage{amsthm}
\usepackage[english]{babel}
\usepackage{amssymb}
\usepackage{graphics}
\usepackage{epsfig}
\usepackage{amscd}
\usepackage{amsfonts}
\usepackage{amsbsy}
\usepackage{float}
\usepackage{caption}
\usepackage{subcaption}
\usepackage{overpic}
\usepackage{dsfont}

\newcommand{\R}{\mathds{R}}

\newtheorem {theorem} {Theorem}
\newtheorem {prop} [theorem] {Proposition}

\newtheorem {lemma} [theorem] {Lemma}
\newtheorem {definition} {Definition}
\newtheorem {remark} {Remark}

\usepackage{color}

\begin{document}

\title[hip-hop  solutions]{Periodic oscillations in a $2N$-body problem}

\author[Oscar Perdomo, Andr\'es Rivera, John A. Arredondo, Nelson Castañeda, ]
{Oscar Perdomo $^1$, Andr\'es Rivera $^2$, John A. Arredondo $^3$, Nelson Castañeda $^4$, }

\address{$^1$ Departament of Mathematics. Central Connecticut State University. New Britain. CT 06050. Connecticut, USA.}
\address{$^2$  Departamento de Ciencias Naturales y Matem\'aticas.
	Pontificia Universidad Javeriana Cali, Facultad de Ingenier\'ia y Ciencias,
	Calle 18 No. 118--250 Cali, Colombia.}
\address{$^3$ Departamento de  Matem\'aticas
	Fundación Universitaria Konrad Lorenz, Facultad de Ciencias e Ingenier\'ia,
	Cra 9 bis 62-43, Bogotá, Colombia.}
\address{$^4$ Departament of Mathematics. Central Connecticut State University. New Britain. CT 06050. Connecticut, USA.}
\email{perdomoosm@cssu.edu, amrivera@javerianacali.edu.co,
	alexander.arredondo@konradlorenz.edu.co, castanedaN@ccsu.edu
}\emph{}



\date{}
\dedicatory{}
\maketitle

\begin{center}\rule{0.9\textwidth}{0.1mm}
\end{center}
\begin{abstract}

\noindent
Hip-Hop solutions of the $2N$-body problem are solutions that satisfy at every instance of time, that the $2N$ bodies with the same mass $m$, are at the vertices of two regular $N$-gons, each one of these $N$-gons are at planes that are equidistant from a fixed plane $\Pi_0$ forming an antiprism. In this paper, we first prove that for every $N$ and every $m$ there exists a family of periodic hip-hop solutions. For every solution in these families the oriented distance to the plane $\Pi_0$, which we call $d(t)$, is an odd function that is also even with respect to $t=T$ for some $T>0.$ For this reason we call solutions in these families, double symmetric solutions. By exploring more carefully our initial set of periodic solutions, we numerically show that some of the branches stablished in our existence theorem have bifurcations that produce branches of solutions with the property that the oriented distance function $d(t)$ is not even with respect to any $T>0$, we call these solutions single symmetry solutions.  We prove that no single symmetry solution is a choreography. We also display explicit double symmetric solutions that are choreographies.


\end{abstract}

\begin{center}\rule{0.9\textwidth}{0.1mm}
\end{center}




\begin{center}
	\includegraphics[width=10cm]{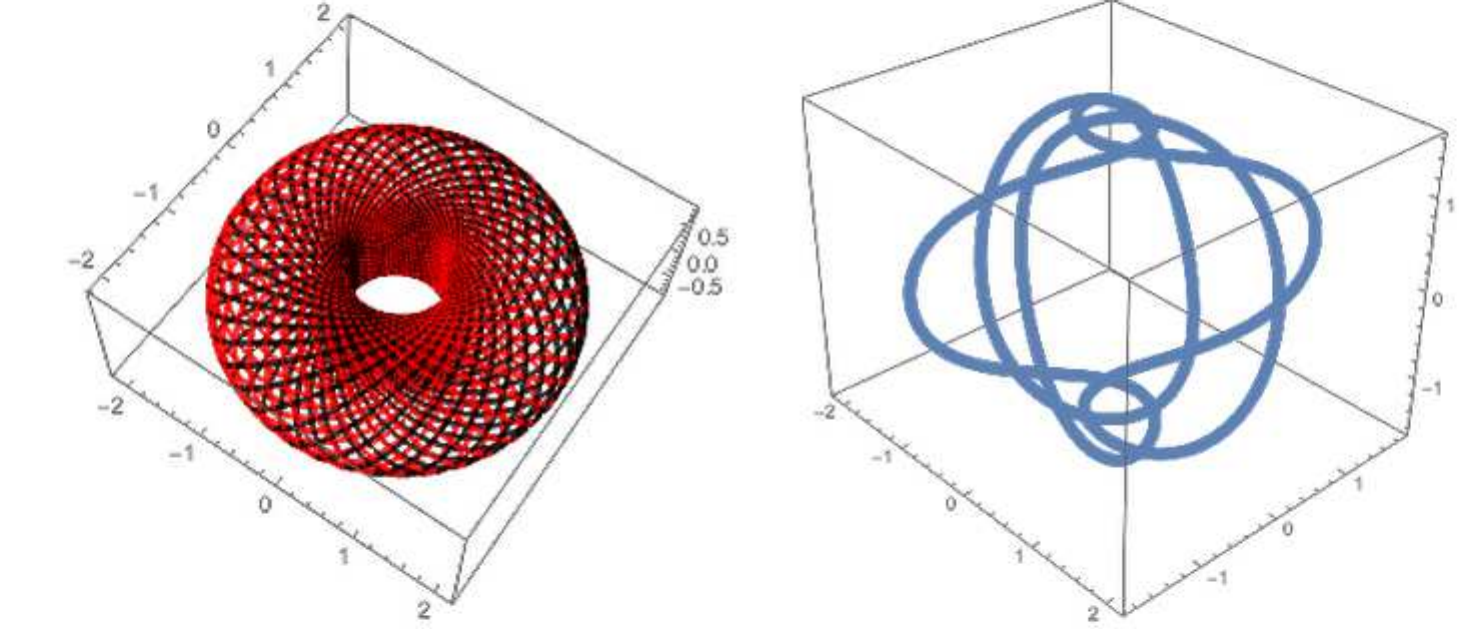}
\end{center}

\section{Introduction}\label{intro}

Hip-hop solutions of the $2N$-body problem are solutions satisfying that: (i) The masses of all the bodies are the same. (ii)  At every instance of time $t$, $N$ of the bodies are at the vertices of a regular $N$-gon contained in a  plane $\Pi_1(t)$ and the other $N$ bodies are at the vertices of a second regular  $N$-gon that differ from the first $N$-gone  by a translation and a rotation of $\frac{2\pi}{N}$ radians and it is contained in a plane $\Pi_2(t)$. (iii) For every $t$, the planes   $\Pi_1(t)$ and  $\Pi_2(t)$ are parallel and they are equidistant from a fixed plane  $\Pi_0$. (iv) The center of the $N$-gones are always in a fixed line $l_0$ perpendicular to $\Pi_0$. In particular 
when  $\Pi_1(t)\ne\Pi_2(t)$ the $2 N$ bodies form an antiprism, and when $\Pi_1(t)=\Pi_2(t)$ the $2 N$ bodies form a regular $2N$-gon. For any hip-hop solution we can define the function $d(t)$ that gives the oriented distance from the plane $\Pi_0$ to the plane $\Pi_1$ and the function $r(t)$ that gives the  distance from the line $l_0$ to any of the bodies. We will assume that the line $l_0$ is the $z$-axis and the plane $\Pi_0$ is the $x$-$y$ plane.  

\begin{figure}[h]
\centerline
{\includegraphics[scale=0.19]{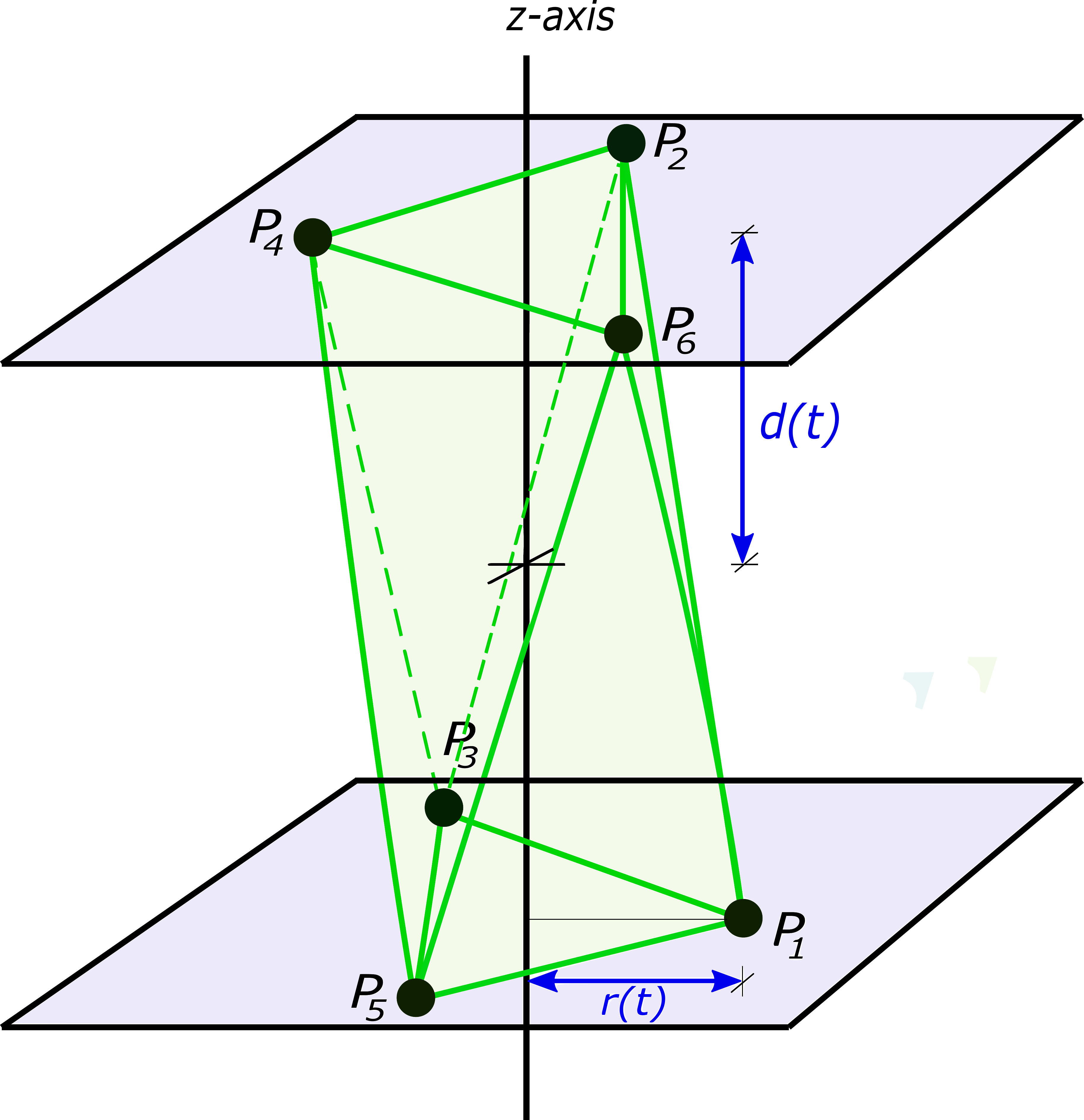}}
\caption{Six bodies moving at the vertices of a regular antiprism all time. The function $r(t)$ is the distance of each body respect to the $z$-axis passing through the center of the two parallel planes where each group of three bodies moves. The function $d(t)$ is the distance of the plane $\Pi_1$ to the plane $\Pi_0=\left\{(x,y,z):z=0\right\}$.}
\label{fig:Fig1}
\end{figure}

We will call a hip-hop solution, \textit{double symmetric}, if the function $d(t)$ satisfies that for all $t$, $d(-t)=-d(t)$ and $d(T-t)=d(T+t)$ for some nonzero $T$. We will call a hip-hop solution, \textit{single symmetric}, if the function $d(t)$ satisfies that $d(-t)=-d(t)$ for all $t$  and there is not a $T$ such that,  $d(T-t)=d(T+t)$  for all  $t$.

The differential equations for hip-hop solutions were provided by Meyer and Schmidt in 1993 \cite{MR1205666} as a model for braided rings of a planet based on a previous model for the rings of Saturn. They called these solutions alternating solutions and the only difference with the hip-hop solutons is that the alternating solutions consider an additional motionless body in the center of mass of the $2N$ bodies. For the alternating solutions, the body in the center makes the role of the planet while the $2N$ bodies are the small bodies conforming the ring. One of the pioneer works that showed the periodicity of some hip-hop solutions is the work of Chenciner and Venturelli \cite{MR1820355}, where the authors used variational methods to prove existence of periodic solutions. Some other papers that show the periodicity of hip-hop solutions by variational methods are \cite{MR1820355, MR2401905, MR2031430}. It is well known that the periodic solutions of the two body problem can be extended to periodic planar solutions of the $n$ body problem, with the trajectory of each body being a rotation of a single ellipse. The link \url{https://youtu.be/RjlZpqDFsDM} leads to a video showing some of these periodic solutions. Using a Bolzano argument, in \cite{MR2570295} the authors show the periodicity of hip-hop solution with motion close to ellipses with high eccentricity. In \cite{MR2267950} the authors use a Poincar\'e analytical continuation method to show the existence of double symmetric hip-hop solutions with trajectories close to circles.

Similar to the work in \cite{MR2267950} our solutions emerge from planar periodic solutions where all the bodies move in the same circle. We study both, doubly and single symmetric solutions. The technique that we use is slightly different from the regular Poincar\'e analytical method. We can say that it is more direct. Let us give the idea behind the method. We consider the functions $r(t)$ and $d(t)$ defined above as functions of one of the initial conditions and the angular momentum, more precisely we consider $R(a,b,t)=r(t)$ and $D(a,b,t)=d(t)$ where $b=\dot{d}(0)$ and $a$ is the angular momentum. All our solutions satisfy $d(0)=0$. For the circular solutions $b=0$ because the solution are planar, $d(t)\equiv 0$, and  the value of $a=a_0$ can be easily found  from the  condition that $r(t)$ is constant. In the Euclidean space $(a,b,T)$ with $T$ a fixed value of time, we have that due to the symmetries of the ordinary differential equations that govern the hip-hop solutions, we have that points in the space that satisfy the system of equations $D_t(a,b,T)=R_t(a,b,T)=0$ produce double symmetric hip-hop solutions with angular momentum $a$ and period $4T$, while points $(a,b,T)$ that satisfy $D(a,b,T)=R_t(a,b,T)=0$ produce hip-hop solutions with angular momentum $a$ and period $2T$ that are potentially single symmetric. With this idea in mind we have that double and single symmetric solutions are obtained from studying three surfaces: The surface in the space $(a,b,T)$ that satisfies $R_t(a,b,T)=0$, the  surface in the space $(a,b,T)$ that satisfies $D_t(a,b,T)=0$ and the surface  in the space $(a,b,T)$ that satisfies $D(a,b,T)=0$. Recall that the intersection of two surfaces is in general a curve and that if the gradient of the functions that define two surfaces are linearly dependent at a particular point in the intersection, then the solution of the system  may be the union of two curves. As an example, Figure \ref{SSP} show points $(a,b,T)$ that solve the system  $D(a,b,T)=R_t(a,b,T)=0$. In this paper we notice that the circular solutions of the $2N$ body problem produces a line $\Gamma=\{(a_0,0,T)\in \mathbb{R}^3\}$ in the space $(a,b,T)$ that solve the systems $D(a,b,T)=R_t(a,b,T)=0$ and $D_t(a,b,T)=R_t(a,b,T)=0$. By computing the gradient of the functions $D$, $D_t$ and $R_t$ we find points in $\Gamma$ that potentially allows nearby points that solve both systems. Half of these (bifurcation) points were found in the paper  \cite{MR2267950} using the Poincare analytical method. The other half that we found were those  bifurcation points that potentially may produce single symmetric solutions. It turned out that points in the smooth curves $\Omega_i$ emanating from the new bifurcation points, did not produced single symmetric hip-hop solutions near the circular solutions but, after doing analytic continuation, we found out that for some values of $N$,  $\Omega_i$ has a bifurcation point that generates a curve of points $\Phi$ in the space $(a,b,T)$ with the property that each of its point represents a single symmetric solution. We proved that an interesting difference between the single and double symmetric families of hip-hop solutions of the $2N$ body problem is the fact that no single symmetric solution can produce a choreographic while there are infinitely many choreographies in the family of double symmetric hip-hop solutions.


%
%


\section{Preliminary results}

The anti-prism $2N$-body problem will be characterized in the following concrete framework:
\noindent
Consider $ Q_ {1}, Q_ {2}, \ldots, Q_ {2 N} $ bodies of equal mass $ m> 0 $, located on the vertices of a regular anti-prism. If $\mathbf{r}_{j}(t)$ is the position of the body $Q_{j}, j=1,\ldots, 2N$ at each instant $t$ and satisfy
\[
\mathbf{r}_{j}(t)=\mathcal{R}^{j-1}\mathbf{r}_{1}(t), \quad j=1,\dots, 2N,
\]
where $\mathcal{R}$ is a rotation/reflection matrix given by
\begin{equation*}
\label{matriz de A-}
\mathcal{R}=
\begin{pmatrix}
\cos (\frac{\pi}{N})& -\sin(\frac{\pi}{N}) & 0 \\
\sin(\frac{\pi}{N}) & \cos (\frac{\pi}{N}) & 0 \\
0& 0& -1
\end{pmatrix}.
\end{equation*}

Introducing cylindrical coordinates $(r,\theta,d)$ 
%
it is shown that the equations of motions of the bodies are given by
\begin{equation}\label{cilindricas}
\begin{cases}
\begin{split}
\ddot{r} &= \frac{a^2}{r^3} - 2rm f(r, d),\\
\ddot{d} &=-\frac{m\,d}{2} g(r, d), \\
\dot{\theta} &=a/r^2,
\end{split}
\end{cases}
\end{equation}

where $\displaystyle{a}$ is the angular momentum of the system and
\begin{equation*}
\begin{split}
f(r, d) &= \sum_{k=1}^{2N-1} \frac{\sin ^2(k \pi/ 2N)}{\left[4r^2 \sin^2(k \pi / 2N) + ((-1)^k - 1)^2 d^2\right]^{3/2}}, \nonumber \\
g(r, d) &= \sum_{k=1}^{2N-1} \frac{((-1)^k - 1)^2}{\left[4r^2 \sin^2(k \pi / 2N) +((-1)^k - 1)^2d^2\right]^{3/2}}.
\end{split}
\end{equation*}

From now on, for $r_0, m$ and $N$ fixed, we denote by 
\[
R(a,b,t)=r(t), \quad D(a,b,t)=d(t), \quad \text{and} \quad \Theta(a,b,t)=\theta(t),
\] 

the solutions of the system (\ref{cilindricas}) with initial conditions
\begin{equation}\label{initial conditions}
	r(0)=r_0,\quad \dot{r}(0)=0,\quad d(0)=0,\quad \dot{d}(0)=b,\quad \theta(0)=0.
\end{equation}

It is clear that $r(t)$, $d(t)$ solves the \textit{reduced problem}
\begin{equation}
	\begin{cases}
		\begin{split}\label{cilindricas 2}
			\ddot{r} &= \frac{a^2}{r^3} - 2rm f(r, d),\\
			\ddot{d} &=-\frac{m\,d}{2} g(r, d),
		\end{split}
	\end{cases}
\end{equation}

then $r(t)$, $d(t)$ and $\theta(t)$ solves (\ref{cilindricas}) if
\[
\theta(t)=\int_{0}^{t}\frac{a}{r^{2}(s)}ds.
\]
Notice that if for some $a,b \in \R$ the functions $R(a,b,t)$ and $D(a,b,t)$ have the same period $T>0$ then, they provide a periodic hip-hop solution of (\ref{cilindricas}) if and only if $\theta(a,b,T)$ is equal to $p\,\pi/q$ with $p$ and $q$ whole numbers. In general, $T$-periodic solutions of the systems (\ref{cilindricas 2}) define \textit{reduced-periodic} solutions of the $2N$-body problem (\ref{cilindricas}). That is, solutions with the property that every $T$ units of time, the positions and velocities of the $2N$-bodies only differ by an rigid motion in $\R^{3}.$

Let us present some useful results on the existence of symmetric periodic solutions of $(\ref{cilindricas})$. To this end, let $r_{0},m\in \R^{+}$ 
and consider the following initial value problem 
\begin{equation}\label{ivp}
\begin{cases}
\begin{split}
\ddot{r} &= F(a,r,d), \quad r(0)=r_{0}, \quad \dot{r}(0)=0,\\
\ddot{d} &=	G(r, d) \quad \quad d(0)=0, \quad \dot{d}(0)=b.\\
\end{split}
\end{cases}
\end{equation}
with
\[
	F(a,r,d) = \dfrac{a^2}{r^3} - 2rm f(r, d), \quad \text{and} \quad	G(r, d)  = -\frac{m d}{2} g(r, d).
\]
%
Recall that, we are denoting the solutions of (\ref{ivp}) as $t\to R(a,b,t)$ and $t \to D(a,b,t)$.  Moreover, it is easy to check the following symmetries: 
\[
F(a,r,d)=F(a,r,-d),\quad \text{and} \quad G(r,d)=-G(r,-d)
\]

%


%
From previous symmetries it is clear that $R(a,b,t)$ is even and $D(a,b,t)$ is odd. Moreover, we have the following result.
\begin{lemma}\label{lem1}
Let $t \to R(a,b,t)$ and $t \to D(a,b,t)$ be a solution of (\ref{ivp}).
\begin{itemize}
	\item[$\triangleright$] If for some $0<T$ we have 
	\[
	\textbf{(I)} \quad R_{t}(a,b,T)=0, \quad D_{t}(a,b,T)  = 0,
	\]
    then $R(a,b,t)$, $D(a,b,t)$ are even functions respect to the line $t=T$. Moreover, both functions are $4T$-periodic.
    \item [$\triangleright$] If for some $0<\tilde{T}$ we have 
    \[
    \textbf{(II)} \quad R_{t}(a,b,\tilde{T})=0, \quad D(a,b,\tilde{T})  = 0,
    \]
    then with respect to the line $t=\tilde{T}$, $D(a,b,t)$ is an odd function  and $R(a,b,t)$ is an even function. Moreover, both functions are $2\tilde{T}$-periodic.
\end{itemize}
\end{lemma}

\begin{definition} We say that a  solution of \eqref{ivp} is of \textit{type I} (\textit{type II}) if it satisfies condition $\textbf{(I)}$ (condition \textbf{(II)}) in Lemma \ref{lem1}. We also call the solutions of type I \underline{double symmetry solutions}.
\end{definition}

\begin{remark} If $(a,b,T)$ solve system \textbf{(I)} then $(a,b,2T)$ solve system \textbf{(II).}
\end{remark}

\vspace{0.2 cm}
For a given $N>1$ define the parameters
\begin{equation}\label{sumas}
	\alpha_{N}=\frac{1}{16}\sum_{k=1}^{2N-1} \frac{((-1)^{k}-1)^{2}}{\sin^{3}(\frac{k\pi}{2N})}, \quad \gamma_{N}= \frac{1}{4}\sum_{k=1}^{2N-1} \frac{1}{\sin(\frac{k\pi}{2N})}.
\end{equation}


\begin{prop}\label{bound}
	For any integer $N>1$, the sum $\alpha_{N}$ and $\gamma_{N}$ defined in \eqref{sumas} satisfy
	\[
	\frac{\gamma_{N}}{\alpha_{N}}\leq \frac{4+\sqrt{2}}{8}.
	\]
\end{prop}
\begin{proof}
	This results is a direct consequence of Lemma 1 in \cite{MR2570295}.
\end{proof}

\begin{remark}\label{rem 2}
 A direct computation shows that $D(a,0,t)=0$ for all $a,t\in \R.$ Moreover, if $a=a_{0}:=\sqrt{m\gamma_{N}r_0} $ then  $R(a_{0}, 0, t) =r_{0}$ for all $t \in \R$. Consequently, the points $(a_0,0,T)$ for all $T\in \R$ satisfies the two systems of equations  in three variables $\textbf{(I)}$ and $\textbf{(II)}$. We will call these solutions the \underline{trivial solutions} of $\textbf{(I)}$ and $\textbf{(II)}$.

\end{remark}



\section{Existence of periodic solutions}
In this section we state and prove the main theoretical result of this document.

%
%

\begin{theorem}\label{main 1}
 Let $N>1$ and $m,r_{0}>0$ fixed. Then there exists $\hat{b}>0$, and a pair of functions  $T, a:]-\hat{b},\hat{b}[ \to \R$, with 
 \[
 T(0)=\frac{\pi}{2}\sqrt{\frac{r_{0}^3}{m\alpha_{N}}}, \quad \text{and} \quad a(0)=\sqrt{m\gamma_{N}r_0},
 \]
such that $\displaystyle{R(a(b), b, t)=r(t)}$ and $D(a(b), b, t)=d(t)$  are $4T$-periodic functions. Moreover, the points $(a(b),b,T(b))$ solve the system \textbf{(I)}.
\end{theorem}
 
\begin{proof}
For fixed $m,r_0$, let $R(a,b,t)=r(t)$, $D(a,b,t)=d(t)$  be the solutions of \eqref{cilindricas 2}. By Lemma \ref{lem1} if for some $a,b$ and $T$ we have 
\begin{equation}\label{zeros}
R_{t}(a,b,T)=0, \quad \text{and} \quad D_{t}(a,b,T)=0,
\end{equation}
then $r(t)$ is even respect to $T$-axis whereas $d(t)$ is an odd function and both are $4T$-periodic. Let $\zeta(t)=(a_{0},0,t)$, $t\in \R.$ From Remark \ref{rem 2} it follows directly that
\[
R_{t}(\zeta(t))=0, \quad\text{and} \quad   D_t(\zeta(t))=0, \quad \forall t\in \R.
\]
In order to study possible bifurcations points in the system \eqref{zeros}, we search a particular value $t=T_{0}$ such that the set of vectors
\[
\Big\{\nabla R_{t}(\zeta(T_{0})), \nabla D_t(\zeta(T_{0}))\Big\},
\]
are linearly dependent. To this end, since $D(a,b,t)$ satisfies the initial value problem
\begin{equation}\label{equation d}
D_{tt}=-\frac{m D}{2}g(R, D), \quad D(0)=0, \quad \dot{D}(0)=b,
\end{equation}
from the existence and uniqueness theorem of ordinary differential equations we deduce that $\displaystyle{D(a,0,t)=0}$ for all $(t,a)$. From here, 
\begin{equation}\label{factorization}
D(a,b,t)=b\hat{D}(a,b,t), \qquad \forall (a,b,t)
\end{equation}
and
\begin{equation}\label{factorization}
D_t(a,b,t)=b\hat{D}_t(a,b,t), \qquad \forall (a,b,t)
\end{equation}
for some appropriate function $\hat{D}$. Moreover, direct computations shows that 
\[
D_{t}(\zeta(t))=D_{a}(\zeta(t))=0, \quad D_{b}(\zeta(t))=\hat{D}(\zeta(t)), \quad \forall t\in \R. 
\]
\[
D_{bt}(\zeta(t))=\hat{D}_{t}(\zeta(t)), \quad \forall t\in \R.
\]
Taking the partial derivative with respect to $b$ on both sides in \eqref{equation d}, the function $D_{b}(\zeta(t))$ satisfies the initial value problem
\[
\ddot{x}+\frac{m \alpha_{N}}{r_{0}^{3}}x=0, \quad x(0)=0, \quad \dot{x}(0)=1.
\]
Therefore, 
\begin{equation}\label{partialD equation}
D_{b}(\zeta(t))=\Big(\frac{m\alpha_{N}}{r^{3}_{0}}\Big)^{-1/2}\sin\Big[\Big(\frac{m\alpha_{N}}{r^{3}_{0}}\Big)^{1/2 }t\Big],
\end{equation}
and
\[
D_{bt}(\zeta(t))=\hat{D}_{t}(\zeta(t))=\cos\Big[\Big(\frac{m\alpha_{N}}{r^{3}_{0}}\Big)^{1/2 }t\Big]. 
\]
in consequence, 
\[
D_{t}(\zeta(T_{0}))=0, \quad \text{if} \quad T_{0}=\frac{\pi}{2}\sqrt{\frac{r^{3}_{0}}{m\alpha_{N}}}.
\] 
This shows that $\displaystyle{\nabla D_{t}(\zeta(T_{0}))=\textbf{0}}$. The previous computations, suggest to study the solutions of the system 
\[
 R_{t}(a,b,T)=0, \quad \text{and} \quad \hat{D}_{t}(a,b,T)=0,
\]
around the point $(a_{0},0,T_{0})$. 

Firstly, we need to compute  $R_{a}(\zeta(t))$, $R_{at}(\zeta(t))$ and $R_{b}(\zeta(t))$. For this purpose, let us recall that 
\[
R_{tt}=F(a, R, D), \quad R(0)=r_{0}, \quad R_{t}(0)=0.
\]
In consequence, taking the partial derivative with respect to $a$ on  both sides, direct computations shows that the function $u(t)=R_{a}(\zeta(t))$ satisfies the equation
\[
\ddot{u}=F_{a}(a_{0}, 0, r_{0})+ F_{R}(a_{0}, 0, r_{0}) u,
\]
where
\begin{equation}
\begin{split}
F_{R}(a, R, D) & =-3 \frac{a^{2}}{R^{4}}-2 m\left[f(R, D)+R f_{R}(R, D)\right], \\
F_{a}(a, R, D) & =\frac{2a }{R^{3}}.
\end{split}
\end{equation}
From here, direct computations shows that
\[
f(r_{0},0)=\frac{\gamma_{N}}{2r_{0}^{3}}, \quad \text{and} \quad f_{R}(r_{0},0)=\frac{-3\gamma_{N}}{2 r_{0}^4},
\]
with $\gamma_{N}$ given in \eqref{sumas}. Therefore, $u(t)$ satisfies the initial value problem,
%
but $\displaystyle{a_0^2=m r_{0}\gamma_N}$, therefore
\[
\ddot{u}+w^{2} u=2w/r_{0}, \quad u(0)=0, \quad \dot{u}(0)=0,
\]
with $\displaystyle{w^{2}=m\gamma_{N}/r_0^3}$, where the solution is given by
\[
R_{a}(\zeta(t))=\frac{2}{w r_{0}}(1-\cos (\omega t)),
\]
implying that
\[
R_{at}(\zeta(t))=\frac{2}{r_{0}} \sin(\omega t).
\]
%
%
%
%

In the same fashion, the function $v(t)=R_{b}(\zeta(t))$ satisfies the differential equation
\[
\ddot{v}=F_{D}(a_{0},r_{0},0) D_{b}(\zeta(t))+F_{R}(a_{0},r_{0},0)v.
\]
Since $r(0)=r_{0}$ and $\dot{r}(0)=0$, the function $v(t)$ is the solution of the initial value problem
\[
\ddot{v}+w^2v=0, \quad v(0)=0, \quad \dot{v}(0)=0.
\]
In consequence,
\[
R_{b}(\zeta(t))=0, \quad \text{and} \quad R_{bt}(\zeta(t))=0, \quad \forall t.
\]

Finally, it follows directly that
\begin{equation*}
R_{tt}(\zeta(t)) =F\left(r_{0}, 0,a_{0}\right)=\frac{a_{0}^{2}}{r_{0}^{3}}-\frac{m\gamma_{N}}{r_{0}^{2}}=0 \quad \forall t.
\end{equation*}
In brief, the previous calculation provide for $\nabla R_{t}$ at $\zeta(T_0)$ that
\begin{equation}\label{nabla dot R}
\nabla R_{t}\big(\zeta(T_{0})\big) =\big(\frac{2}{r_{0}} \sin(\omega T_{0}), 0,0\big), 
\end{equation}

%
where $\displaystyle{\omega T_{0}=(\pi\sqrt{\gamma_{N}/\alpha_{N}}})/2$. From Proposition \ref{bound} it follows that $\omega T_{0}\neq p\pi$ for every positive integer $p$ implying that $\nabla R_{t}(\zeta(T_{0}))$ does not vanish.
Finally, we use the relation \eqref{factorization} to compute $\displaystyle{\nabla \hat{D}_{t}(\zeta(T_{0}))}$. Notice that, 
\[
D_{bt}(\zeta(t))=\hat{D}_{t}(\zeta(t)), \quad \text{and} \quad D_{ba}(\zeta(t))=\hat{D}_{a}(\zeta(t)).
\]
Then, from \eqref{partialD equation} we have
\[
\hat{D}_{t}(\zeta(T_{0}))=-1, \quad \text{and} \quad \hat{D}_{a}(\zeta(T_{0}))=0.
\]
Moreover
\begin{equation}\label{nabla D tilde}
\nabla \hat{D}_{t}(\zeta(t))=\Big(D_{tba}(\zeta(t)),\frac{1}{2}D_{bbt}(\zeta(t)),D_{btt}(\zeta(t))\Big).
\end{equation}
%

%
%

%


%
%
%
%
To sum up, from expressions \eqref{nabla dot R} and \eqref{nabla D tilde} we obtain
%
%
\[
\Lambda:=\nabla R_{t}(\zeta(T_{0})) \times \nabla \hat{D}(\zeta(T_{0}))=\big(0, -2D_{btt}(\zeta(T_{0}))r_{0}^{-1}\sin(\omega T_{0}), D_{bbt}(\zeta(T_{0})) r_{0}^{-1}\sin(\omega T_{0})\big).
\]

Since,
\[
D_{btt}(\zeta(t))=\hat{D}_{tt}(\zeta(t))=-\Big(\frac{m\alpha_{N}}{r^{3}_{0}}\Big)^{1/2 }\sin\Big[\Big(\frac{m\alpha_{N}}{r^{3}_{0}}\Big)^{1/2 }t\Big],
\]
we have $\displaystyle{D_{btt}(\zeta(T_{0}/2))=-\pi/2T_{0}}.$ This shows that the second component of $\displaystyle{\Lambda}$ is different from zero. Then, by the Implicit Function Theorem, there exists $\hat{b}>0$, and a unique pair of functions  $T, a:]-\hat{b},\hat{b}[ \to \R$, such that 
\[
b\to T(b), \quad b\to a(b),
\]
for $b\in ]-\hat{b}, \hat{b}[$, with $T(0)=T_0$ and $a(0)=a_0$,
%
%
such that
\[
R_{t}(\lambda(b))=0, \quad \hat{D}_{t}(\lambda(b))=0,
\]
%
%
with
$$
\begin{aligned}
\lambda:]-\hat{b}, & \hat{b}\left[\longrightarrow \mathbb{R}^{3}\right.\\
& b \longrightarrow \lambda(s)=(a(b),b,T(b)).
\end{aligned}
$$
Therefore, for each $b\in ]-\hat{b}, \hat{b}[$ it follows
\[
R_{t}(\lambda(b))=0, \quad \text{and} \quad  D_{t}(\lambda(b))=b\hat{D}_{t}(\lambda(b))=0.
\]
By Lemma \ref{lem1} we get that for any $b\in ]-\hat{b}, \hat{b}[$ the functions
\[
r(t)=R(a(b),b,t), \quad \text{and} \quad d(t)= D(a(b),b,t),
\]
provides a $4T(b)$-periodic solution of the reduced problem (\ref{cilindricas 2}) with $d(t)$ an odd function,  whereas $r(t)=R(a(b),b,t)$ and $d(t)=D(a(b),b,t)$ are both even respect to the line $t=T(b)$. 
\end{proof}

Due to Remark 1 we have the following result:

\begin{theorem}\label{main 2}
	Let $N>1$ and $m,r_{0}>0$ fixed. Then there exists $\hat{b}>0$, and a pair of functions  $T, a:]-\hat{b},\hat{b}[ \to \R$, with 
	\[
	T(0)=\pi\sqrt{\frac{r_{0}^3}{m\alpha_{N}}}, \quad \text{and} \quad a(0)=\sqrt{m\gamma_{N}r_0},
	\]
	such that $\displaystyle{R(a(b), b, t)=r(t)}$ and $D(a(b), b, t)=d(t)$  are a $2T$-periodic functions. Moreover, the points $(a(b),b,T(b))$ solve the system \textbf{(II)}.
\end{theorem}


%
%

\section{Branches emanating from the bifurcations}

In this section we use analytic continuation to extend the branches emanating form the bifurcations points

$$p_0=\left(\, \sqrt{m \gamma_N r_0},\, 0,\, \frac{\pi}{2} \sqrt{\frac{r_0^3}{m\alpha_N}}\, \right)\quad \hbox{and}\quad q_0=\left(\, \sqrt{m \gamma_N r_0},\, 0,\, \pi \sqrt{\frac{r_0^3}{m\alpha_N}}\, \right),$$

obtained in Theorem \ref{main 1} and Theorem \ref{main 2}.


Notice that the solutions from $p_{0}$, $q_{0}$ are essentially the same, all of them have double symmetry. With the purpose of searching  for solutions of type $\textbf{(II)}$ that are not of type $\textbf{(I)}$ we extend the branch starting at the point $ q_0$ by applying analytic continuation and a method similar to the one presented in \cite{MR2267950} to system \textbf{(II)}. Taking $N=3$, $r_0=2$ and $m=1$ we have

$$q_0=\left( \sqrt{\frac{5}{2}+\frac{2}{\sqrt{3}}},0,4 \sqrt{\frac{2}{17}} \pi\right) \approx (1.91173,0,4.31023).$$


 The analytic continuation method give us a table that we labeled DSP because, as explained above, near the point $q_0$, the solutions satisfy the double symmetry property.  We found a  bifurcation point --that we labeled $B$-- along this branch that gave us a new branch with single symmetric hip-hop solutions. Recall from the introduction that a solution is called single symmetric if $d(t)$  is even and there is not $T\neq 0$ such that $d(T-t)=d(T+t)$ for all $t$. Before we address the existence of this bifurcation point $B$, let us elaborate on some properties of the DSP branch.

Figure \ref{DSP} shows the table $ DSP =\{Q_1=q_0, Q_2,\dots, Q_{5806}=q_f\}.$ It is a list of points
in the $ a,b, T $ space that satisfies the conditions
$$     |R_{t}(Q_i)|<10^{-4}  \ \   |F(Q_i)|<10^{-4}   \ \ \   \text{for each $ Q_i \in DSP $}.    $$

For the solutions along the DSP branch the system of particles reaches the maximum vertical expansion (that is expansion in the $z$-direction) at the same time that it reaches a maximum contraction towards the $z$-axis. In other words, when the particles are at their maximum height they are also the closest they can be to the $z$-axis. This is illustrated by the joint graph of $ d(t) $ and $ r(t) $ for the points $ Q_{210}$ and $Q_{4225}$   shown in Figure \ref{selectedpointsinDSP}.

\begin{figure}[hbtp]
	\begin{center}\includegraphics[width=.8\textwidth]{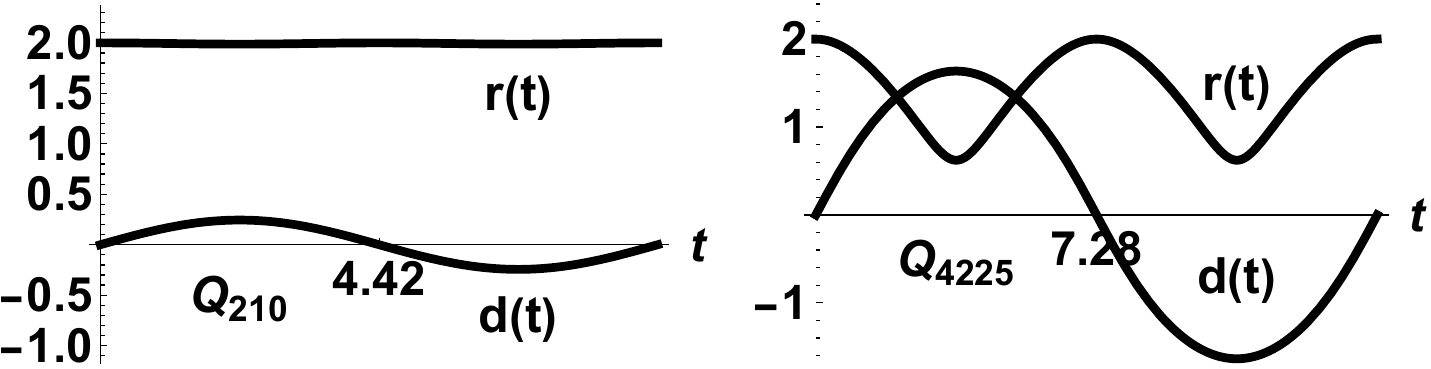}
		\caption{The graph of the functions $d(t)$ and $r(t)$ for two solutions in the branch DSP. Since the point $Q_{210}$ is close to the point $q_0$ then $r(t)$ is almost constant and the function $d(t)$ does not oscillate much. Recall that the solution $q_0$, $d(t)$ vanishes and $r(t)$ are constant. Both solutions show that after a quarter of a period of the function $d(t)$,  $r(t)$ reaches a minimum while $d(t)$ reaches a maximum. For this reason these solutions have an additional symmetry.} \label{selectedpointsinDSP}
	\end{center}
\end{figure}

\begin{figure}[hbtp]
	\begin{center}\includegraphics[width=.35\textwidth]{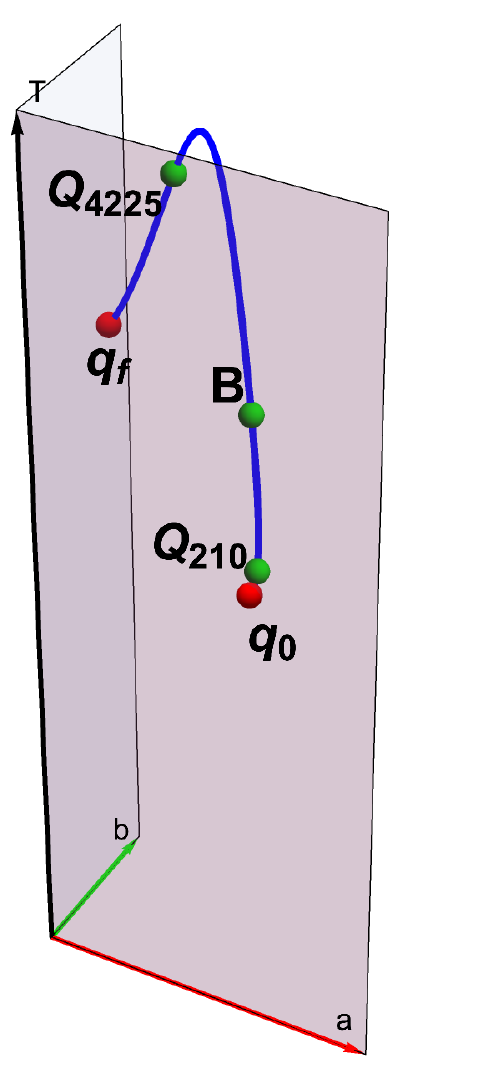}
		\caption{Points in the space $a$, $b$, $T$ found by doing analytic continuation near the bifurcation point $q_0$ provided by Theorem \ref{main 1}. All these points are associated with solutions that have double symmetry and for this reason we have labelled this family of solutions as DSP. We show the bifurcation point $B$, and the points $Q_{210}$ and $Q_{4225}$. We will show later that: $(i)$ we will have another family of solutions with not double symmetry emanating from the point $B$, $(ii)$ the point $Q_{210}$ will represent a solution where bodies 1, 3 and 5 share the same trajectory and $(iii)$ bodies 2, 4 and 6 also share the same trajectory. For the solution represented by the point $Q_{4225}$, all the six trajectories are different. } \label{DSP}
	\end{center}
\end{figure}

The branch DSP  ends at the point
$$q_f=(0.259786, 0.780202, 5.80955),$$
a point with angular moment $a$ near zero,  close to a solution that represents a collision, with three of the bodies colliding at maximum height. 
\begin{figure}[hbtp]\label{cg}
	\begin{center}\includegraphics[width=.35\textwidth]{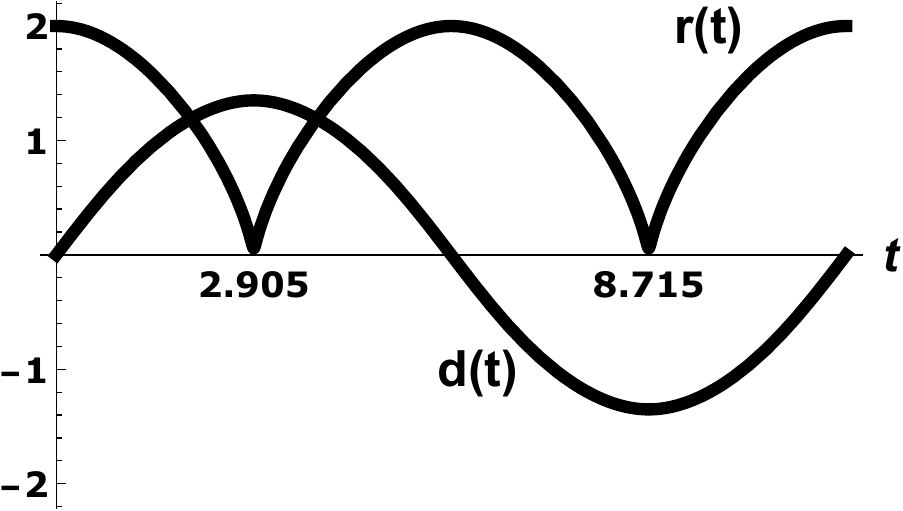}
		\caption{This is the graph of the functions $d(t)=D( 0.259786, 0.780202, t)$ and $ r(t)=R(0.259786, 0.780202, t)$ which is the solution associated with point $q_f=Q_{5806}$ located at the end of the branch that starts at the point
			$q_0$. As we can see, this solution has double symmetry, both functions have derivatives vanishing at $\frac{0.580955}{2}$. We can also see how this solution is near a solution with a collision of $3$ bodies near the point $(x,y,z)=(0,0,1.39)$  and the collision of the other $3$ bodies near the point $(x,y,z)=(0,0,-1.39).$}
	\end{center}
\end{figure}
The name for the set $DSP$ stands for ``double symmetry points'' because each one of these points represent a solution with period $2T$ with $d $ and $r$ even with respect to $t=T/2$, $r$ even with respect to $t=T$ and $d$ odd with respect to $t=T$.

\subsection{Bifurcation point along the DSP branch}

By the implicit function theorem, we have that as long as the vectors $\nabla D(a^0,b^0,T^0)$ and  $\nabla R_{t}(a^0,b^0,T^0)$ are linearly independent for points $(a^0,b^0,T^0)$ that satisfy

\begin{equation}\label{sys1}
	\begin{cases}
		\begin{split}
			R_{t}(a,b,T)=& \, 0,\\ 
		F(a,b,T)=& \, 0,
	\end{split}
\end{cases}
\end{equation}

then, the solution of the system (\ref{sys1}) is given by a smooth curve (not bifurcation points near that point) near $(a^0,b^0,T^0)$. We have noticed that at the point $B=(1.34958, 0.727361, 7.05373)$, which is one of the points in the set $DSP$, satisfies that

$$\nabla D(B)=(4.58986, 17.2712, -0.727943),\quad \nabla R_{t}(B)=(1.44703, 5.44591, -0.229381),$$

and

$$\nabla D(B)\times \nabla R_{t}(B)=(0.0026399, -0.000527769, 0.0041234).$$

The information above provided numerical evidence of the possible existence of a point in the branch where the gradients $\nabla D(a,b,T)$ and  $\nabla R_{t}(a,b,T)$ are linearly dependent. Therefore we search for solution of system \eqref{sys1} near $B$ but away from the smooth curve suggested by the points in the set $DSP$. Indeed we were able to find a point that satisfy the system \eqref{sys1}, away from the trajectory of the points in DSP but near the point $B$. After doing analytic continuation to this new point, we were able to find the collection of points

$$SSP=\{W_1, W_2,\dots, W_{14154}\},$$

all of them satisfying $|R_{t}(W_i)|<10^{-4}$ and $|F(W_i)|<10^{-4}$. The name $SSP$ stand for ``single symmetry points'' because the solutions associated with these points do not satisfy that $R_{t}(a,b,T/2)=0$ and $F(a,b,T/2)=0$. Figure \ref{SSP} shows how the two branches and the point $B$.

Therefore, we have numerically found solution with only one symmetry. 

\begin{figure}[hbtp]
	\begin{center}\includegraphics[width=.15\textwidth]{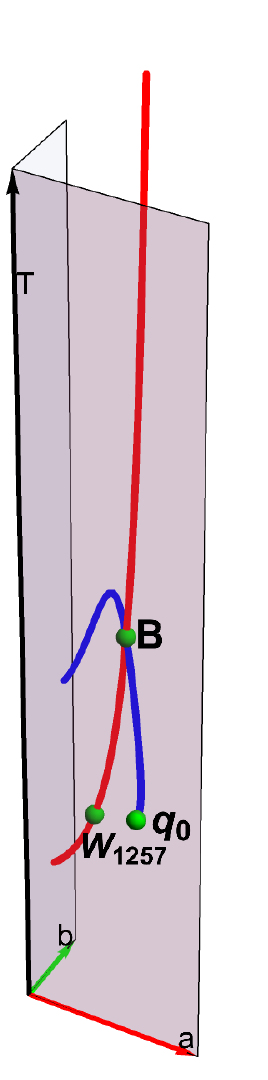}
		\caption{Points of the two branches SSP and DSP. The branch DSP starts at $q_0$ while the branch SSP emanated from point that satisfies the system (\ref{sys1}) near $B$ and away from the trajectory DSP. We have highlighted the point $W_{1257}$ because this point represent a periodic solution where the bodies 1, 3 and 5 share the same trajectory and the bodies 2, 4 and 6 also share the same trajectory.} \label{SSP}
	\end{center}
\end{figure}

The solutions that correspond to  the  SSP  branch are characterized by the fact that the maximum contraction toward the $z$-axis occurs when the system is still expanding (for the points before the point of bifurcation)  or already contracting (for the points after the point of bifurcation)  in the vertical direction. This is illustrated in Figure \ref{gdrssp}.



\begin{figure}[hbtp]
	\begin{center}\includegraphics[width=.8\textwidth]{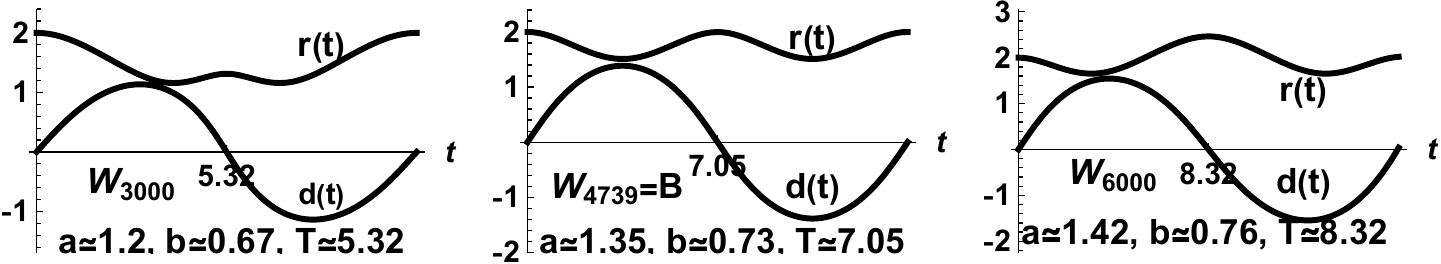}
		\caption{The image on the left shows the graph of the solutions $r(t)$ and $d(t)$ associated with the point $W_{3000}\in$ SSP. We can see how the maximum of $d(t)$ is obtained before the minimum of the function $r$. This means that when the bodies reaches the maximum vertical distance, they continue to move closer to the $z$-axis. The image in the center shows the graph of the solutions $r(t)$ and $d(t)$ associated with the bifurcation point. These two functions have double symmetry. The image on the right shows the graphs of the solutions $r(t)$ and $d(t)$ associated with the point $W_{6000}\in$ SSP.   We can see how the maximum of $d(t)$ is obtained after the minimum of the function $r$. } \label{gdrssp}
	\end{center}
\end{figure}

\subsection{On the number of trajectories and choreographic solutions}

In theory, the six bodies of the periodic solutions can follow $6, 3, 2, $ or one single trajectory. In the last case, when all the bodies follow  the same orbit, the solution is called a choreography. In this subsection, we show examples of the  four cases.

If we label the bodies 1, 2, 3, 4, 5, 6 according to the order of their projections on the $x$-$y$ plane in the direction of rotation about the $z$-axis then the number of trajectories is determined by which of the starting positions will be reached first by the body number 1. More precisely,

Let  $ 1 \rightarrow j $ denote the fact that for some instant of time $t > 0 $  the condition  

\[ C(t,k):     r_1(t)  =  r_k(0) , \ \text{and} \ \   \dot{r}_1(t) = \dot{r}_k(0),\]

is met by $k = j$ and that for every $ 0 < \tau < t  $ and all $ k \in \{1, 2,..., 6\}$   the condition  $  C(\tau, k) $ is false.

Then the number of trajectories are:

Six in the case  $  1  \rightarrow  1 $

One in the case  $  1  \rightarrow  2  $   or  $  1  \rightarrow 5.$

Two in the case  $ 1  \rightarrow  3  $

Three in the case  $  1 \rightarrow 4.$
\subsubsection{Number of orbits for solutions on the branch $DSP$.}
For the solutions in the branch $DSP$ we have that the number of trajectories of  the solution depends on $\theta_0=\theta(T)$: the angle of rotation of the solution after half of the common period of the functions $r(t)$ and $d(t)$. Since for every solution on this branch, the function $r$ is even with respect to $t=T/2$, then we have that $r(T)=r(0)=2$. We also have that for any integer $k$,  $\theta(kT)=k\theta_0$. Notice that when the solution is periodic, we can find three integers $k,j$ and $l$ that satisfy the following integer equation

\begin{eqnarray}\label{integereq}
	IE(k,j,l): k\theta_0=j\frac{\pi}{3}+2\pi l.
\end{eqnarray}

Recall that $d(kT)=0$ for every integer $k$ and $\dot{d}(kT)<0$ if $k$ is odd and $\dot{d}(kT)>0$ if $k$ is even. Also notice that $k\theta_{0}=j\frac{\pi}{3}+2\pi l$ for some integer $k,j$ and $l$ indicates that the body starting at $(r_{0},0,0)$ moves, after $kT$ units of time to the initial position of the body starting at $(r_{0}\cos(j\frac{\pi}{2}),r_{0}\sin(j\frac{\pi}{2}),0)$. Each body rotates $j\frac{\pi}{3}$ radians during their first $kT$ units of time.

For a periodic solution, let us denote $k_0$ smallest positive integer such that $IE(k_0,j,l)$ is satisfied for a pair of integers $j$ and $l$ with $j>0$ and $k_0+j$ even. We denote by $j_0$ the smallest positive integer with $j_0+k_0$ even such that $IE(k_0,j_0,l)$ is satisfied for some integer $l$. We have that if $j_0$ is $1$ or $5$, then the solution is a choreography, this is, we have only one trajectory that is share by the six bodies. If $j_0$ is $2$ or $4$ then, there are two trajectories. If $j_0$ is $3$, then there are three trajectories and if $j_0$ is $6$ then, there are $6$ trajectories.

As explain above,  the angle $\theta_0(T)$ plays an important role determining the number of trajectories. Figure \ref{thetaDSP} shows how the angle $\theta_0$ changes for different values of $T$ in the branch $DSP$. Since the first point in $DSP$ is near a trivial solution, we can compute the starting value for these curve. In this particular case, it is $\sqrt{\frac{1}{51} \left(4 \sqrt{3}+15\right)} \pi\approx 2.06$.
\begin{figure}[hbtp]
	\begin{center}\includegraphics[width=.35\textwidth]{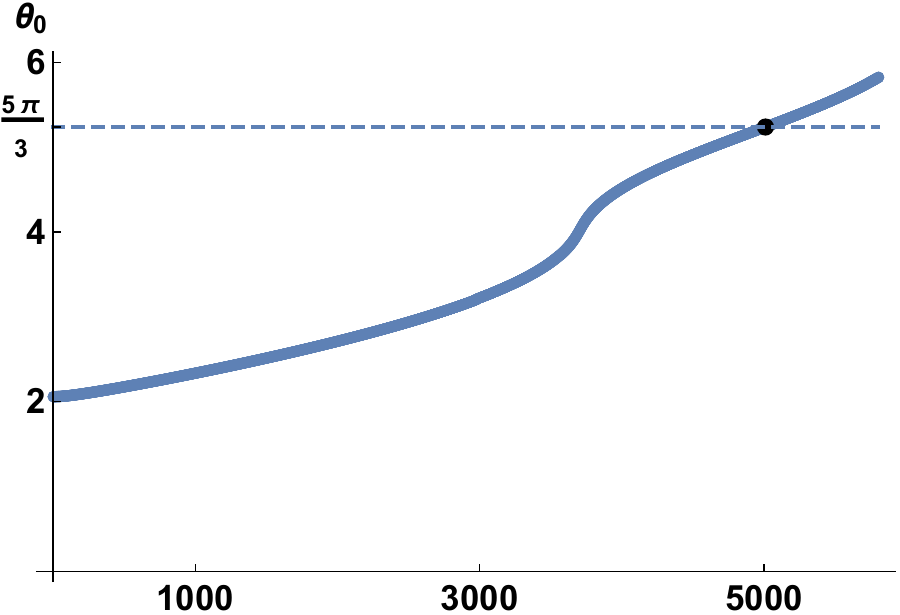}
		\caption{Graph of the points $(i,\theta_0(T_{i}))$ where $T_{i}$ the third entry of $Q_{i}$ in the branch $ DSP =\{Q_1=q_0, Q_2,\dots, Q_{5806}=q_f\}$. This graph points out the fact that $5\pi/3$ is in the range of the function $\theta_{0}$. } \label{thetaDSP}
	\end{center}
\end{figure}

As pointed out in Figure \ref{thetaDSP}, $5\pi/3$ is in the range of the function $\theta$ as a function of $T$. When $T$ moves along the smooth curve of points $(a,b,T)$ extending the points in DSP. The closest value of $\theta(T)$ to $5\pi/3$ in the branch $DSP=\{Q_1\dots,Q_{5806}\}$ happens for the solution $5010$, we have that $\theta(T_{5009})<5\pi/3<\theta(T_{5010})$ and $\theta(T_{5010})-5\pi/3<0.00033$. The intermediate value theorem shows that there is choreographic in the family of solution with double symmetry. Let us explain the orbit of this choreographic, assuming for practical reasons that it is given by the solution 5010 in the branch DSP. The initial configuration of the six bodies is displayed in Figure \ref{initialpos}

\begin{figure}[hbtp]
	\begin{center}\includegraphics[width=.36\textwidth]{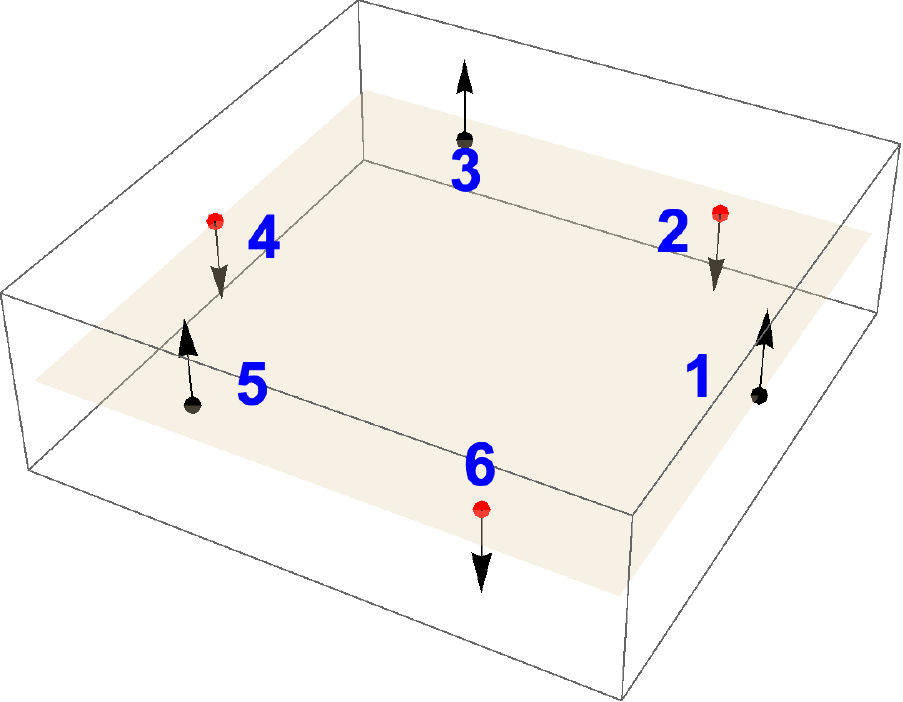}
		\caption{Labelling of the bodies. Initially they form a regular dodecagon.}\label{initialpos}
	\end{center}
\end{figure}

The arrows indicate if the bodies start going up or down. Notice that those bodies that go up, they do not do it vertically due to the rotation motion that all of them are doing. The point $Q_{5010}=(a,b,T)=(0.581691, 0.810807, 6.53465)$, this means that after $T=6.53465$ units of time, the body 1 will move to the initial position of body 6. Figure \ref{onetra} shows the rest of the details of this motion. In this case, in terms of the equation $IE(k,j,l)$ in \eqref{integereq} we have that $k_0=1$, $j_0=5$ and $l=5$.

\begin{figure}[hbtp]
	\begin{center}\includegraphics[width=.7\textwidth]{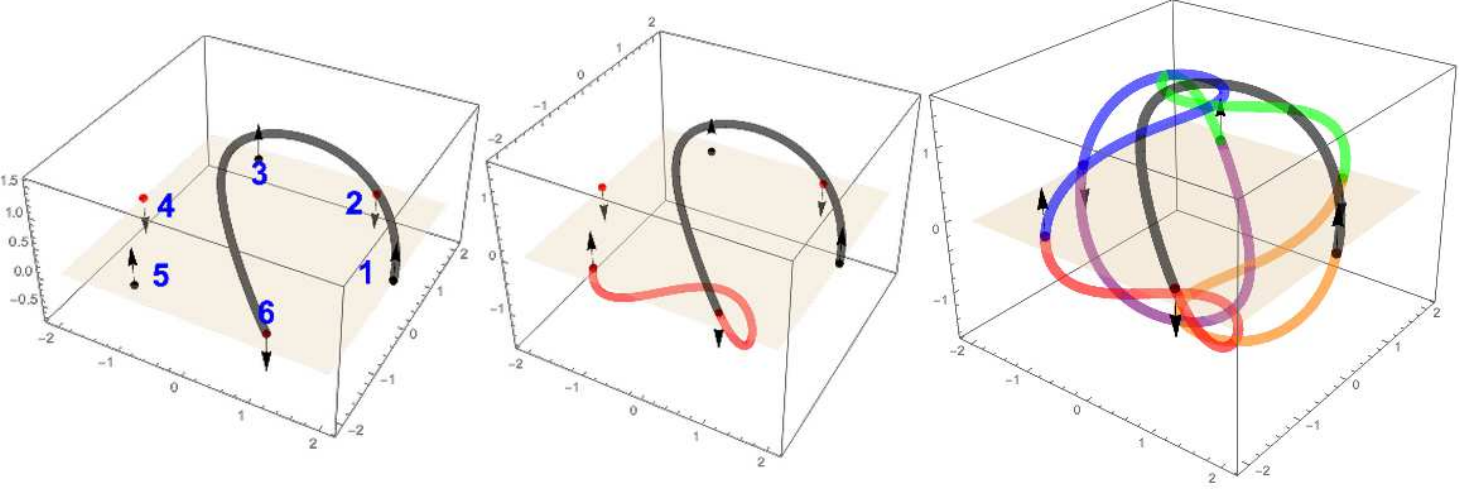}
		\caption{The first image shows the trajectory of the body 1 in the solution represented by $Q_{5010}$ for values of $t$ between $0$ and $T=6.53465$. The second image shows the trajectories of bodies 1 and 6 in the same interval of time. The third image show the trajectories of all 6 bodies in the same interval of time. This third image is also the trajectory of each one of the bodies between $t=0$ and $t=6T$.}\label{onetra}
	\end{center}
\end{figure}


In the same way, $5 \pi/3$ is in the range of $\theta_0(T)$, we can check that $\pi$ is in the range of this function. This time the closest value of $\theta(T)$ to $\pi$ in the branch $DSP$ happens for the solution represented with the point $Q_{2878}$, we have that $0<\pi-\theta_{0}(T_{2878})<0.00007$.  Let us explain the trajectory of this solution, assuming for practical reasons that the solution represented by $Q_{2878}$ in the branch DSP satisfies $\theta_0(T)=\pi$. The point $Q_{2878}=(1.37188, 0.717167, 6.95687)$, this means that after $T=6.95687$ units of time, the body 1 will move to the initial position of body 4. The rest of the orbits are explained in Figure \ref{theetra}.

\begin{figure}[hbtp]
	\begin{center}\includegraphics[width=.7\textwidth]{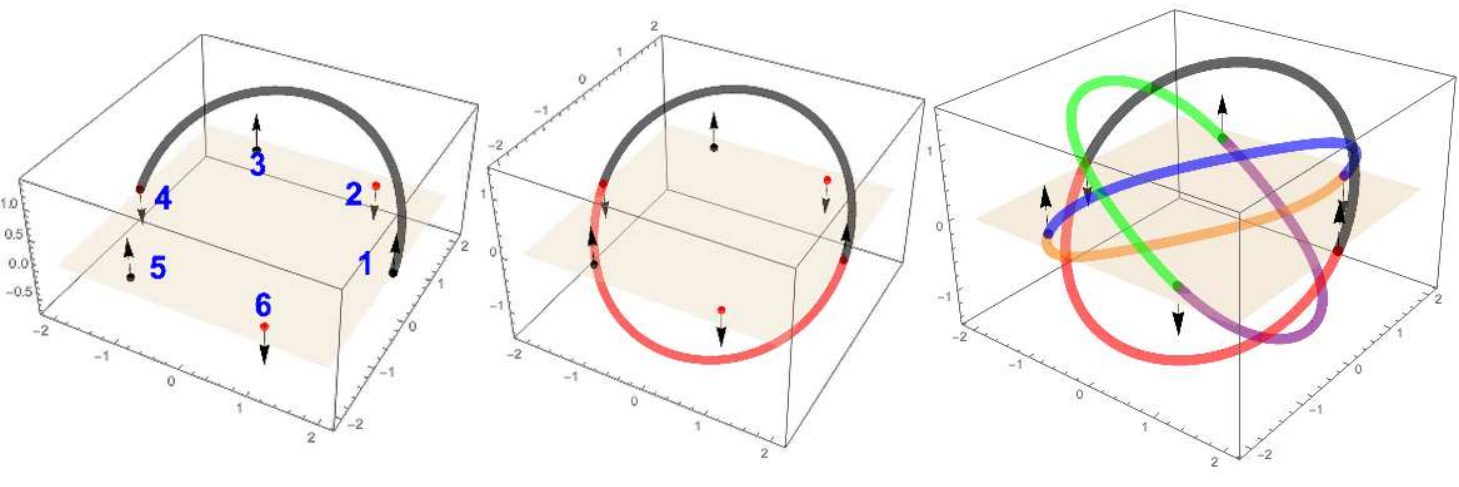}
		\caption{The image on the left shows the trajectory of the body 1 in the solution represented by $Q_{2878}$ for values of $t$ between $0$ and $T_{2878}=6.95687$ and the second image shows the trajectory in the same interval for the bodies 1 and 4. This second image represents the whole orbit for the first and fourth body. They share the same orbit. Bodies 2 and 5 share the same trajectory and the bodies 3 and 6 also share the same orbit. The third image shows the three trajectories.}\label{theetra}
	\end{center}
\end{figure}

The solution represented by $Q_{210}=(1.88461, 0.175173, 4.41712)$ provides a solution that has two trajectories. We have that $\theta_{0}(T_{210})\approx \frac{2\pi}{3}$. In this case after $T_{210}=4.41712$ the body one reaches the initial position of body 3 but it does not reach it with the right direction of the velocity, after $2T$ units of time, the body 1 reaches the initial position of body 5 with the same initial velocity as well. For this solution $k_0=2$, $j_0=4$ and $l=0$, see Figure \ref{twotra}.

\begin{figure}[hbtp]
	\begin{center}\includegraphics[width=.6\textwidth]{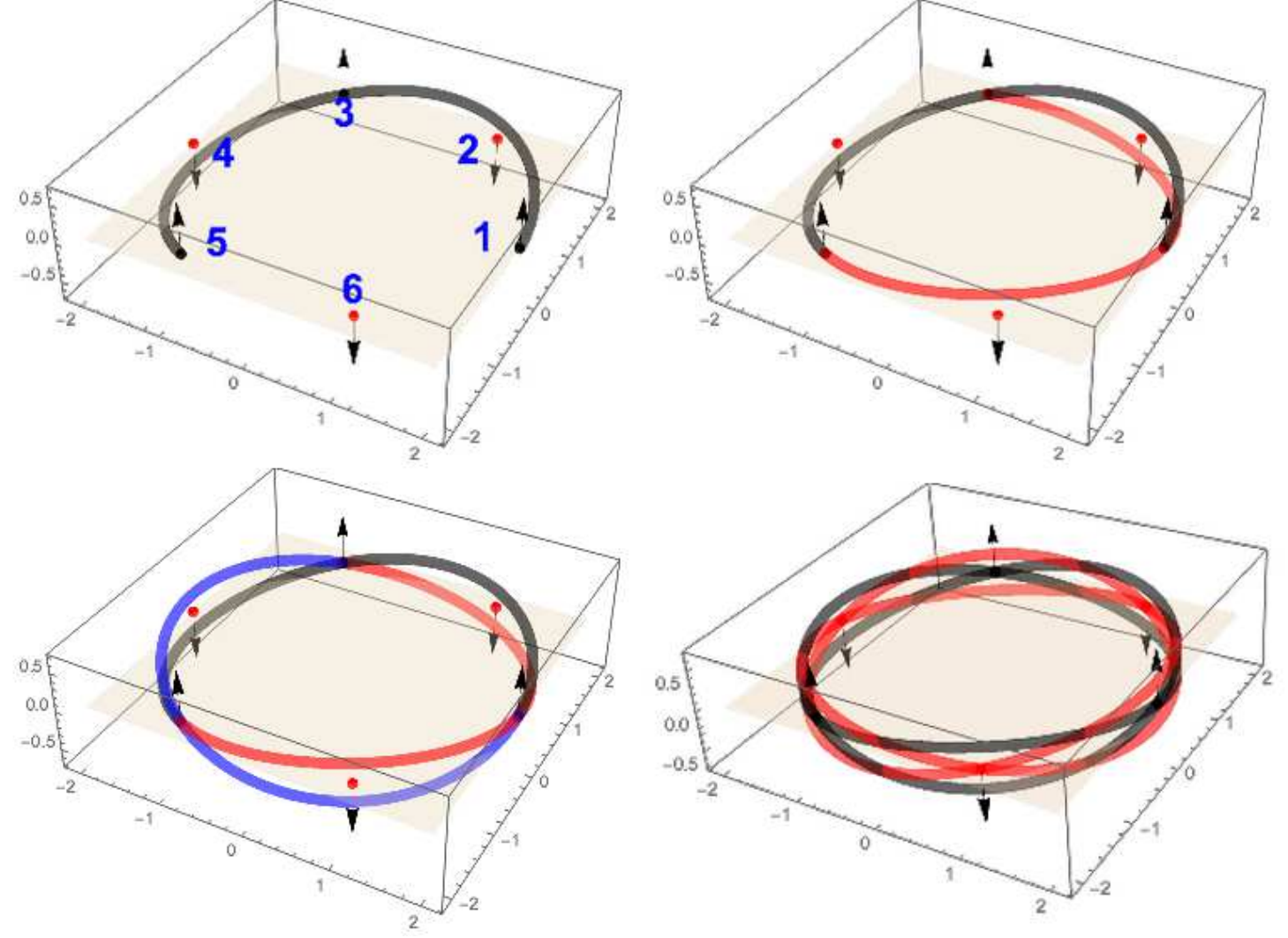}
		\caption{The first image shows the trajectory of the body 1 in the solution represented by $Q_{210}$ for values of $t$ between $0$ and $2T_{210}$ where $T_{210}= 4.41712$, the second image shows the trajectory in the same interval for the bodies 1 and 5 and the third image shows the trajectory in the same interval for the bodies 1, 5 and 3. This third image represent the whole orbit for the first, third and fifth body. They share the same orbit. The bodies 2, 4 and 6 also share the same trajectory. The fourth image shows the two trajectories.}\label{twotra}
	\end{center}
\end{figure}

The solution represented by $Q_{4225}=(0.827163, 0.825182, 7.28011)$ provides a solution that has six trajectories. We have that $\theta_{0}(T_{4225})\approx \frac{3\pi}{2}$. In this case the body one reaches back its  initial position and velocity after $4T_{4225}$ units of time. Moreover, it never reaches the initial position and velocity of the other five bodies.  For this solution $k_0=4$, $j_0=6$ and $l=0$, see Figure \ref{sixtra}.

\begin{figure}[hbtp]
	\begin{center}\includegraphics[width=.78\textwidth]{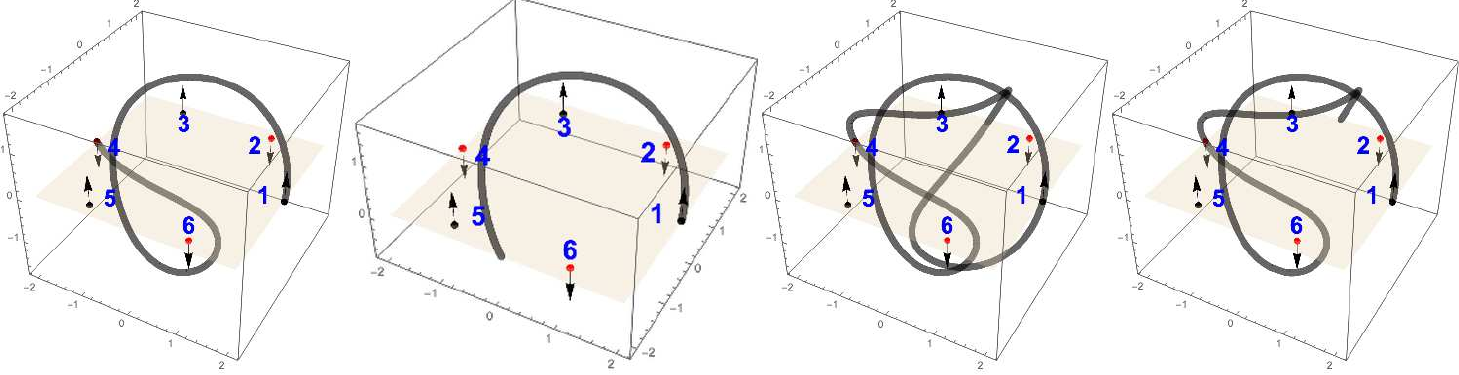}
		\caption{The second image shows the trajectory of the body 1 in the solution represented by $Q_{4225}$ for values of $t$ between $0$ and $T_{4225}=7.28011$. The first image shows the trajectory of the body 1 for values of $t$ between $0$ and $2T_{4225}$. Notice that even though the ending point of this portion of this trajectory is the starting position for the fourth body, the velocity is not the initial velocity of the fourth body. The fourth image shows the trajectory of the body 1 for values of $t$ bet\-ween $0$ and $3T_{4225}$ and the third image shows the trajectory of the body 1 for values of $t$ between $0$ and $4T_{4225}$. }\label{sixtra}
	\end{center}
\end{figure}

\subsubsection{Number of orbits for solutions on the branch SSP}

Before we start this subsection, let us point out the following:

\begin{prop} Let us assume that  a periodic solution  of the Hip-Hop problem is given by the functions $(d(t),r(t),\theta(t))$ satisfying the following condition:
	
	\begin{enumerate}
		\item
		$d(t)$ and $r(t)$ are periodic with period $2T$ for some  $T>0$,
		
		\item
		$d(t)=0$ only for values of $t=kT$ with $k$ an integer, 
		\item
		$r(kT)\ne r((k+1)T)$ for all  $k$ an integer,
		
	\end{enumerate}
	
	then, the solution cannot be a choreography.
\end{prop}

\begin{proof} Let us argue by contradiction. If we have a choreography, then there exist a time $\tilde{T}$ when the body 1 reaches the initial position and velocity of the body 2. By hypothesis (2) $\tilde{T}=kT$ for some integer $k$. On the other hand our two conditions on the function $r$, (period equal to $2T$ and $r(lT)\ne r((l+1)T)$ with $l$ an integer) implies that
	$r(lT)=r(0)$ if $l$ is even and $r(lT)\ne r(0)$ if $l$ is odd. Since $r(kT)=r_0$ then $k$ must be even. Since $d$ has period $2T$ then $\dot{d}(kT)=\dot{d}(0)$. This is a contradiction because the initial velocity of the body 2 goes in the ``opposite'' directions. If body 1 starts going up, the body 2 starts going down. 
\end{proof}

One feature of the solutions of the SSP branch (other than the one that corresponds to the point of bifurcation $B$)  is that any body intersect the plane $z=0$ only for values of $t=kT$ where $k$ is an integer. More importantly, the distance to the $z$-axis of  any of the bodies  changes when $t$ increases changes from $t=kT$ to $t=(k+1)T$, this is, $r(kT)\ne r((k+1)T)$. In particular the condition $ 1 \longrightarrow j $ never occurs if $ j $ is even. Therefore, in the SSP branch the only possible periodic solutions will have two trajectories with the bodies 1, 3, 5 sharing one orbit and  2, 4, 6  the other.  

In the equation \eqref{integereq} when $  1 \rightarrow 3 $ we have $ j = 2 $ and when $  1  \rightarrow 5 $ we have  $ j = 4. $  We set $k = 2q, j = 2p $ in the condition  $ IE(k, j, l) $  

\begin{eqnarray}
	IE(k,j,l): 2 q \theta_0=2 p \frac{\pi}{3}+ 2 \pi l,
\end{eqnarray}
that simplifies to  
\begin{eqnarray}\label{integereq2}
	q\theta_0=p\frac{\pi}{3}+\pi l,
\end{eqnarray}

where $ q $  and $ l $ are positive integers and $ p = 1 $ or $ p = 2. $

Assume  $ p = 1.$  The closest value of $ q \theta(T)$ to $\pi/3 + \pi l $ in the branch $SSP=\{P_1\dots,P_{14154}\}$ happens at  $P_{1257}=(0.886201, 0.557961, 3.61393)$, with $ q = l = 13.$  have that

\[ 13 \theta(T_{1257}) - \left( \frac{\pi}{3} +  13 \pi   \right)  < 0.000102,   \]
and 
\[ 13 \theta(T_{1258}) - \left( \frac{\pi}{3} +  13 \pi   \right)  >  -0.000133.   \]

The intermediate value theorem guarantees  that there is solution with single symmetry that has exactly two orbits with the bodies 1, 3, 5 in one and 2, 4, 6 in the other. The first image on the Figure between the abstract and the Introduction shows both trajectories. For practical purposes we assume that this solution is given by the solution 1257 in the SSP branch.  Figures \ref{ssp1} and \ref{ssp2} explain the trajectories in this particular solution.


%
\begin{figure}[h]
	\begin{center}\includegraphics[width=.75\textwidth]{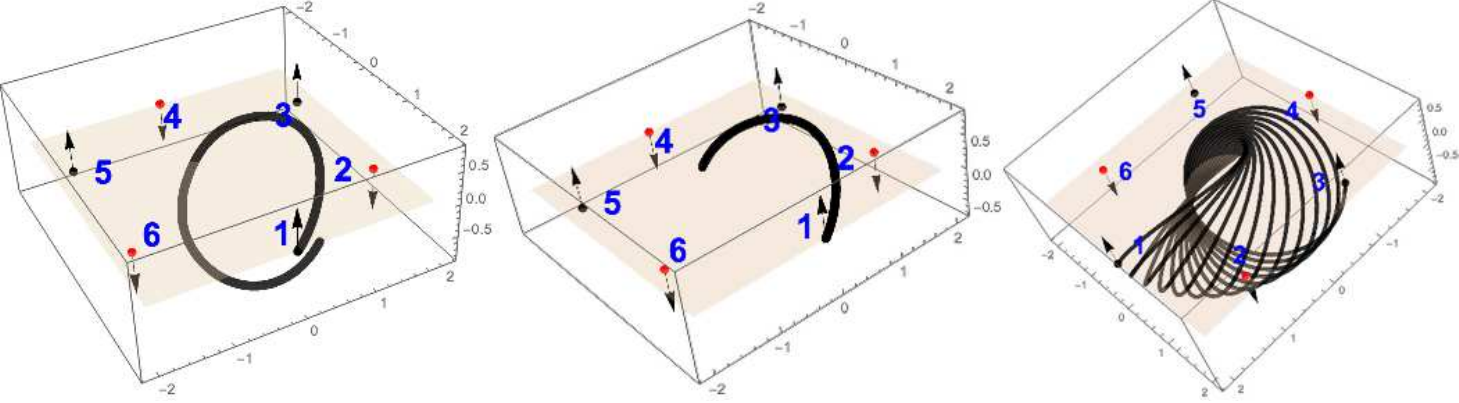}
		\caption{Images for the solution given by the point $P_{1257}$ in $SSP$. The second image shows the trajectory of body 1 from $t=0$ to $t=T_{1257}=3.61393$. Notice that the body comes back to the plane $z=0$ with a value of $r$ different from $2$. The first image show the trajectory of body. From $t=0$ to $t=2T_{1257}$. This time the end position of the body is 2 units away from the $z$-axis. The third image shows the trajectory of body 1 from $t=0$ to $t=26 T_{1257}$. We can see how the trajectory of the body ends at the initial position of body 3.} \label{ssp1}
	\end{center}
\end{figure}

\begin{figure}[h]
	\begin{center}\includegraphics[width=.75\textwidth]{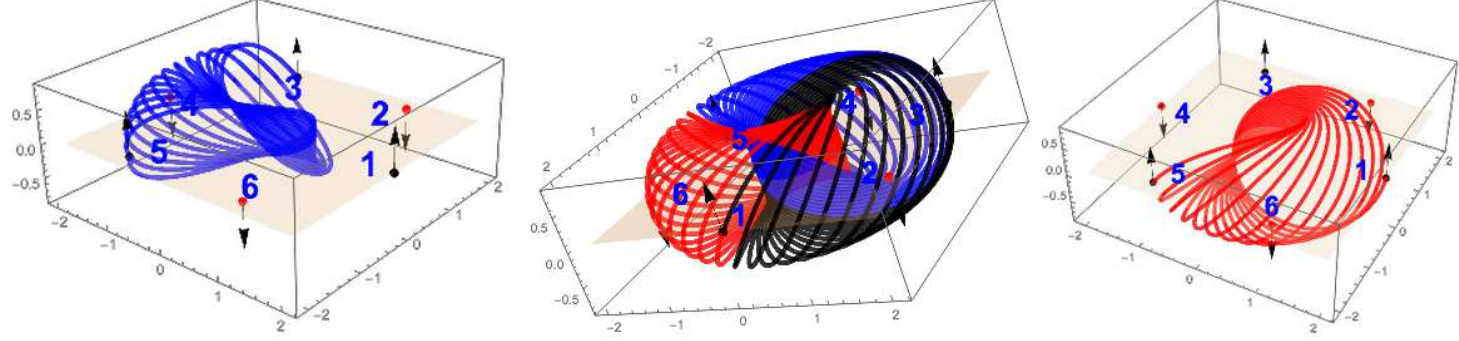}
		\caption{The first image shows the trajectory of body 1 from $t=26 T_{1257}$ to $t=52 T_{1257}$. We can see how the trajectory of the body ends at the initial position of body 5. This trajectory agrees with the trajectory of body 3 from $t=0$ to $t=26 T_{1257}$. The third image shows the trajectory of body 1 from $t=52 T_{1257}$ to $t=78 T_{1257}$. The second image shows the trajectory of body 1 from $t=0$ to $t=78 T_{1257}$. This is a closed embedded curve which is also the trajectory of bodies 3 and 5.} \label{ssp2}
	\end{center}
\end{figure}

\end{document}